\documentclass{article}

\usepackage{mathrsfs}
\usepackage{amsfonts, amssymb, amsmath, amsthm}

\usepackage{enumerate}
\usepackage{graphicx}
\usepackage{tikz}
\usetikzlibrary{calc,intersections,backgrounds,cd}
\usepackage{ytableau}
\usepackage[colorlinks=true]{hyperref}
\usepackage[left=1truein,right=1truein,bottom=1truein]{geometry}
\usepackage{cite}

\usepackage{float,subfig}
\usepackage{color}

\newcommand{\set}[1]{\left\{#1\right\}}

\newcommand{\mb}[1]{\mathbf {#1}}

\usepackage[utf8]{inputenc}
\DeclareFixedFont{\ttb}{T1}{txtt}{bx}{n}{9} 
\DeclareFixedFont{\ttm}{T1}{txtt}{m}{n}{9}  

\usepackage{xcolor}
\definecolor{keyword}{rgb}{0,0,0.5}
\definecolor{emph}{rgb}{0.6,0,0}
\definecolor{string}{rgb}{0,0.5,0}

\usepackage{listings}

\theoremstyle{plain}
\newtheorem{mthm}{Main Theorem}
\newtheorem{thm}{Theorem}[section]
\newtheorem{conj}[thm]{Conjecture}
\newtheorem{prop}[thm]{Proposition}
\newtheorem{cor}[thm]{Corollary}
\newtheorem{lem}[thm]{Lemma}

\theoremstyle{definition}

\title{Towards a Combinatorial Description of the Intersection Product on $\mathbb{P}^{2[N]}$}
\author{Alexander Jon Stathis}
\date{}

\begin{document}

\maketitle

\begin{abstract}
We prove that there is an algorithm to compute the class of the intersection of the divisor of schemes incident to a fixed line with any other class of a basis of the Chow ring $A^*(\mathbb{P}^{2[N]})$ due to Mallavibarrena and Sols. This is progress towards a combinatorial description of the intersection product on the Hilbert scheme of points in the projective plane. 
\end{abstract}

\section{Introduction}

The Hilbert scheme $X^{[N]}$ of $N$ points on a smooth projective surface $X$ and its Chow ring $A(X^{[N]})$ have been extensively studied. Most notably, Nakajima and Grojnowski provided a geometric description of the sum $\oplus_{N \geq 0} H^*(X^{[N]})$ as a representation of a Heisenberg algebra \cite{N,N2,GROJNOWSKI}. However, relatively little is known about the structure of the individual Chow rings $A(X^{[N]})$. The Betti numbers were computed by G\"{o}ttsche \cite{GOTTSCHE}, and there are bases due to Nakajima and Grojnowski via their description. In the case that $X = \mathbb{P}^{2}$, the first known basis is due to Ellingsrud and Str\o mme (ES) \cite{ES,ES2}. In the cases when $N$ is small, the structure of the ring $A(\mathbb{P}^{2[N]})$ has been completely worked out in the ES basis \cite{EL,EL2,EL3}. Descriptions for the equivariant Chow rings have also been worked out \cite{EVAIN}. By degenerating the ES basis, Mallavibarrena and Sols (MS) provided an alternative basis which is better suited to the task of computation \cite{MS}. 

Each class in the MS basis is defined via incidence conditions indexed by a triple of partitions. A description of the MS basis can be found in Section \ref{sec:msbasis}. It is natural to ask if this combinatorial data encodes the intersection product similar to that of the partitions associated to Schubert cells of a Grassmannian. To that end, there is an algorithm to compute the intersection of any two complementary codimension elements of the MS basis \cite{ME}. It is necessarily more complicated than in the case of the Grassmannian, but the algorithm is positive and completely combinatorial. 

There is a natural divisor on $\mathbb{P}^{2[N]}$ which is the locus of schemes incident to a fixed general line -- it is the divisor associated to the triple $(0,(1),(1,\ldots,1))$ in the MS basis. The goal of this paper is to provide an explicit algorithm to compute the intersection of this divisor, which we call $H$, with any element of the MS basis. The main technical work is conducted in Section \ref{sec:degenerations} and Section \ref{sec:coefs}. The difficulty is the computation of the class of natural loci occurring in the intersection via successive degenerations (see Proposition \ref{prop:firstdeg} and Proposition \ref{prop:maindeg}), and determining the multiplicities of the components of the special fiber of these degenerations. These components are described by explicit partitions $\lambda$ obtained from the starting partition $\mb{m}$ by subtracting one from some of its entries, and the multiplicity of each component is determined as a number $c(\mb{m},\lambda)$ depending only on these partitions. An example of this algorithm can be found in Section \ref{sec:example}. In Section \ref{sec:remarks}, we state a purely combinatorial conjecture about the numbers $c(\mb{m},\lambda)$ which would result in an explicit description of the class of these natural loci. 

The other naturally occurring loci in the intersections can be either be immediately expressed in the MS basis or, when this is not the case, can be dealt with via a small modification of a degeneration due to Mallavibarrena and Sols (see Corollary \ref{cor:MS2}). This culminates in our main theorem.

\begin{mthm} \label{thm:main1} Let $H$ be the class of the locus of schemes incident to a fixed general line, and let $\sigma$ be an element of the MS basis. There is an explicit algorithm to compute the class in the MS basis of the intersection $H \cdot \sigma$. \end{mthm}

Finally, an implementation of the degenerations found in Section \ref{sec:degenerations} in Python 3 can be found on the author's website. 

We thank Izzet Coskun and Tim Ryan for many fruitful conversations during the investigation of these results. The author was partially supported by an NSF RTG DMS-1246844 grant during the completion of this work. 

\section{The Basis of Mallavibarrena and Sols} \label{sec:msbasis}

The purpose of this section is to define the basis of Mallavibarrena and Sols \cite{MS}.
 
Fix a triple of partitions $\alpha = (\mb{a},\mb{b},\mb{c})$ of nonnegative integers $A,B,C$ such that $A + B + C = N$, respectively. Let $r,s,t$ be the lengths of $\mb{a},\mb{b},\mb{c}$ and let $e,f,g$ be indices for $\mb{a},\mb{b},\mb{c}$, respectively. We will associate to this data a locally closed subset $U_{\alpha}$ of the open set in $\mathbb{P}^{2[N]}$ of reduced subschemes. The class $\sigma_{\alpha} = \left[\overline{U_{\alpha}}\right]$ of its closure in $\mathbb{P}^{2[N]}$ will be an element of the MS basis. 

Fix $r$ general lines $L_e$ with $r$ general points $P_e \in L_e$, $s$ general lines $M_f$, and a general point $Q$. The locus $U_{\alpha}$ is the locus of reduced schemes $Z$ in $\mathbb{P}^{2[N]}$ which can be written as the disjoint union of three subschemes $Z = Z_1 \cup Z_2 \cup Z_3$ such that:

\begin{enumerate}
\item $Z$ does not contain $Q$ or any point of intersection of any pair of fixed lines.
\item $Z_1$ contains each point $P_e$, and meets each line $L_e$ in $\mb{a}_e$ points.
\item $Z_2$ meets each line $M_f$ in $\mb{b}_f$ points.
\item $Z_3$ contains $t$ disjoint subschemes $Z_{3,g}$ consisting of $\mb{c}_g$ points collinear with $Q$. 
\end{enumerate}

The codimension of $U_{\alpha}$ in $\mathbb{P}^{2[N]}$ is $N + r - t$ so that the class $\sigma_{\alpha}$ is an element of $A^{N + r -t}(\mathbb{P}^{2[N]})$.

\begin{thm}[Mallavibarrena and Sols \cite{MS}] 
The collection of classes $\set{\sigma_{\alpha}}$ is a basis for the Chow ring $A\left(\mathbb{P}^{2[N]}\right)$ as $\alpha$ ranges over all triples of partitions of all triples of nonnegative integers whose sum is $N$. 
\end{thm}

We will often refer to the lines spanned by the subschemes $Z_{3,g}$ as \emph{moving lines} and think of them as lines through the point $Q$ which vary.

\subsection{Some Examples of the Basis}

Consider the triple of partitions $\alpha = (0,(1),(1,\ldots,1))$. We fix a general line $M_1$ and a point $Q$, and consider the subset of schemes $Z$ such that $Z$ can be written as the disjoint union of two subschemes $Z_2$ and $Z_3$ where:
\begin{enumerate}
\item $Z_2$ meets the line $M_1$, and
\item $Z_3$ contains $N-1$ distinct subschemes each of which is collinear with $Q$. 
\end{enumerate}
The codimension of the associated class $\sigma_{\alpha}$ is $N + 0 - (N-1) = 1$. It follows that $\sigma_{\alpha}$ is the locus of schemes which meet a fixed general line in the plane. We will refer to this divisor by $H$ from here on out. 

Consider now the triple of partitions $\beta = (0,0,(2,1,\ldots,1))$. In this case, we fix only a general point $Q$. The locally open set $U_{\beta}$ is the collection of schemes $Z$ such that it contains a subscheme of length two collinear with $Q$. As with $\alpha$, the codimension of the associated class $\sigma_{\beta}$ has codimension $N + 0 - (N-1) = 1$ in the Hilbert scheme. This class, along with $H$, generates the Picard group of $\mathbb{P}^{2[N]}$. 

It is often useful to draw schematic pictures for these classes. Fixed lines and points will be solid, points which are allowed to vary will be hollow, and lines which are moving will be dashed. Points through which the moving lines vary will be marked as a thick ``X".

\section{An Illustrative Example}\label{sec:example}

Let $\alpha = (0,0,(3,2,1))$ be a triple of partitions for $N = 6$ and let $\sigma$ be the associated class in the MS basis for $A(\mathbb{P}^{2[6]})$. Fix a line $L$ and let $H$ be the locus of schemes whose support meets $L$, and fix a point $P$ not on $L$ and let $U$ be the locus of schemes which contain distinct subschemes of length one, two, and three collinear with $P$. See Figure \ref{fig:example0}. 
\begin{figure}[H]
    \centering
    \input{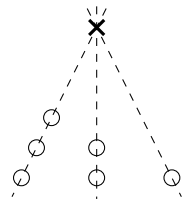}
    \caption{The schematic diagram for $\sigma$.}
    \label{fig:example0}
\end{figure}
The intersection $H \cap U$ consists of three irreducible components: the first corresponding to when the free point of $U$ lies on $L$, the second corresponding to when of the points in the subscheme of length two collinear with $P$ lies on $L$, and finally when one of the points of the subscheme of length three collinear with $P$ lies on $L$. See Figure \ref{fig:intcomps} for the corresponding pictures. We have labeled each component with names that we will define precisely later. 
\begin{figure}[H]
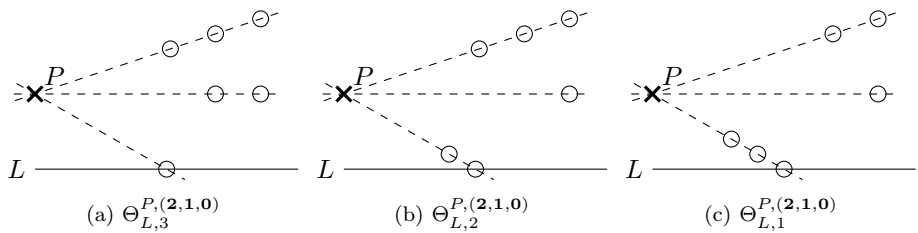
 
    \centering
    
    \subfloat[$\Theta^{P,\mb{(2,1,0)}}_{L,3}$\label{fig:intcomps1}]{\input{figures/pieri/examples/complicated/comp1-3}}
    \subfloat[$\Theta^{P,\mb{(2,1,0)}}_{L,2}$\label{fig:intcomps2}]{\input{figures/pieri/examples/complicated/comp1-2}}
    \subfloat[$\Theta^{P,\mb{(2,1,0)}}_{L,1}$]{\input{figures/pieri/examples/complicated/comp1-1}}

    \caption{The components of the intersection $H \cdot \sigma$.}
    \label{fig:intcomps}
\end{figure}
The first locus is already the class associated to the triple of partitions $(0,(1),(3,2))$ in the MS basis. We will resolve the class of the second and third locus via a series of degenerations. We will start with the third component.

To do this, we degenerate the point $P$ onto the line $L$. There are three components in the limit each consisting of at most one of the points on each moving line colliding with the point $P$. More precisely, they are the locus of schemes with a subscheme of length three varying on $L$ and a subscheme of length two collinear with $P$, the locus of subschemes containing $P$ and two disjoint subschemes of length two collinear with $P$, and the final component consisting of subschemes containing a subscheme of length two supported at $P$ and a subscheme of length two collinear with $P$. See \ref{fig:deg1comps} for the corresponding pictures. The first component is numerically equivalent to the element of the MS basis associated to $(0,(3),(2,1))$.
\begin{figure}[H]
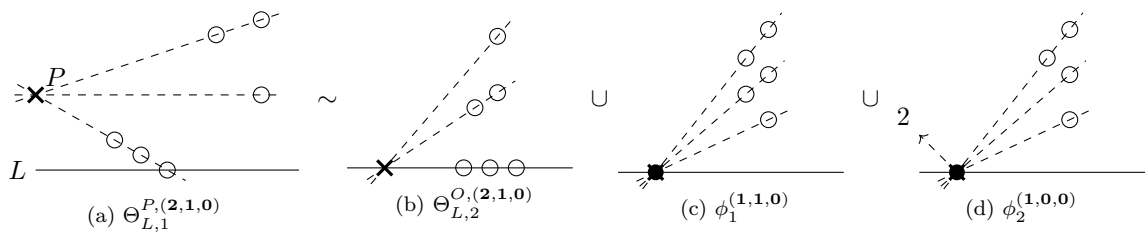
 
    \centering
    
    \raisebox{-.4\height}{\subfloat[$\Theta^{P,\mb{(2,1,0)}}_{L,1}$]{\input{figures/pieri/examples/complicated/comp1-1}}}\ $\sim$
    \raisebox{-.4\height}{\subfloat[$\Theta^{O,\mb{(2,1,0)}}_{L,2}$\label{fig:deg1comps1}]{\input{figures/pieri/examples/complicated/comp2-1}}}\ $\cup$
    \raisebox{-.4\height}{\subfloat[$\phi^{\mb{(1,1,0)}}_{1}$\label{fig:deg1comps2}]{\input{figures/pieri/examples/complicated/comp2-2}}}\ $\cup$
    \raisebox{-.4\height}{\subfloat[$\phi^{\mb{(1,0,0)}}_{2}$\label{fig:deg1comps3}]{\input{figures/pieri/examples/complicated/comp2-3}}}

    \caption{The three loci of the first degeneration.}
    \label{fig:deg1comps}
\end{figure}

To determine the multiplicities of each component, we pair both sides of the degeneration with explicit classes. These classes isolate the multiplicity of each component as the intersection of this class with the general fiber of the degeneration. See Section \ref{sec:coefs} for a more explicit description of the classes, but the idea is that we fix a general point for each moving line in the component and then consider the locus of schemes which contain nonreduced subschemes of the correct length supported at each of these points. The first component always appears with multiplicity equal to the number of points on the fixed line, in this case three. The resulting formula to determine the multiplicities of the other components is combinatorial based on the partition indexing the components, and in this case it is the number of ways to obtain the sequence $\mb{m}$ from the sequence $\lambda$ by adding the correct number of boxes, at most one to each row in $\lambda$, and then reordering. See Figure \ref{fig:counting}. 
\begin{figure}[H]
\centering
\input{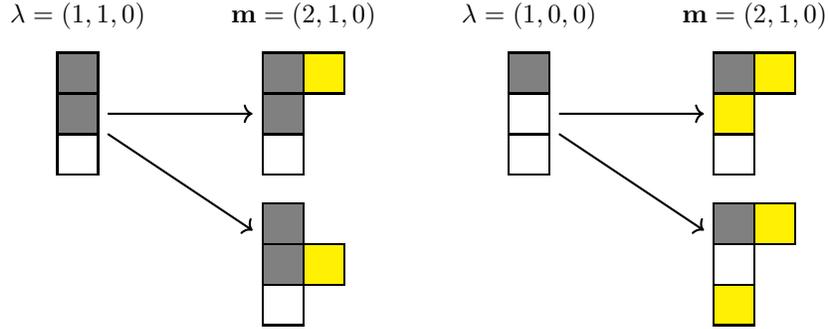}
\caption{The different ways of assembling $(2,1,0)$ by adding a box to $(1,1,0)$ [left] and $(2,1,0)$ from $(1,0,0)$ [right]. The young diagrams in the bottom row require reordering. Gray boxes are entries of $\lambda$, yellow boxes are added, and white boxes indicate zeros in the partitions (which we allow here, see Section \ref{sec:degenerations}).}
\label{fig:counting}
\end{figure}
We record the multiplicities for all three components in Table \ref{tab:deg1ints} below. 
\begin{table}[H]
\centering
\begin{tabular}{ c c c c }
 & $\Theta^{O,\mb{(2,1,0)}}_{L,2}$ & $\phi^{\mb{(1,1,0)}}_{1}$ & $\phi^{\mb{(1,0,0)}}_{2}$ \\ 
 \hline
 Mult. of Comp.   & 3 & 2 & 2
\end{tabular}
\caption{The multiplicities of the components of the degeneration in Figure \ref{fig:deg1comps}.}
\label{tab:deg1ints}
\end{table}
To write the remaining classes in the MS basis, we degenerate again. We start with a locus which is identical except the subscheme of positive length is supported at a fixed point $Q$ distinct from $P$. We degenerate $Q$ onto $P$, and obtain the locus we want as one of the irreducible components. We always start with the loci with the lowest length at $P$ (when multiple loci appear with the same length at $P$, we do them all simultaneously). In our case, we degenerate the second locus (Figure \ref{fig:deg1comps2}) first. Starting with the class of the locus of subschemes containing a fixed general point $Q$ and two subschemes of length two collinear with the fixed general point $P$, we degenerate $Q$ onto $P$. The special fiber of this degeneration is supported along three irreducible components consisting of at most one of the points in each moving subscheme of length two colliding with $P$. Precisely, the first is the locus of schemes containing $P$ and two subschemes of length two collinear with $P$, the second is the locus of schemes containing a length two subscheme supported at $P$ and a subscheme of length two collinear with $Q$, and the third is the locus of schemes containing a length three subscheme supported at $P$. See Figure \ref{fig:deg2comps}. 
\begin{figure}[H]
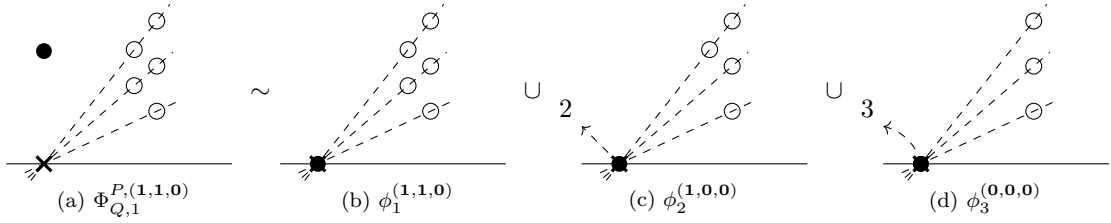

    \centering
    
    \raisebox{-.4\height}{\subfloat[$\Phi^{P,\mb{(1,1,0)}}_{Q,1}$]{\input{figures/pieri/examples/complicated/comp3-0}}}\ $\sim$
    \raisebox{-.4\height}{\subfloat[$\phi^{\mb{(1,1,0)}}_{1}$\label{fig:deg2comps1}]{\input{figures/pieri/examples/complicated/comp2-2}}}\ $\cup$
    \raisebox{-.4\height}{\subfloat[$\phi^{\mb{(1,0,0)}}_{2}$\label{fig:deg2comps2}]{\input{figures/pieri/examples/complicated/comp2-3}}}\ $\cup$
    \raisebox{-.4\height}{\subfloat[$\phi^{\mb{(0,0,0)}}_{3}$\label{fig:deg2comps3}]{\input{figures/pieri/examples/complicated/comp3-3}}}
    
    \caption{The three loci of the degeneration.}
    \label{fig:deg2comps}
\end{figure}
We again determine the multiplicities as the number of ways to add the correct number of 1's to the entries of $\lambda$ to obtain $\mb{m}$ after possibly reordering. For instance, to get $(1,1,0)$ from $(0,0,0)$, we must add two 1's to any of the three 0's in $\lambda$ for a total of three choices. See Figure \ref{fig:counting2}.
\begin{figure}[H]
\centering
\input{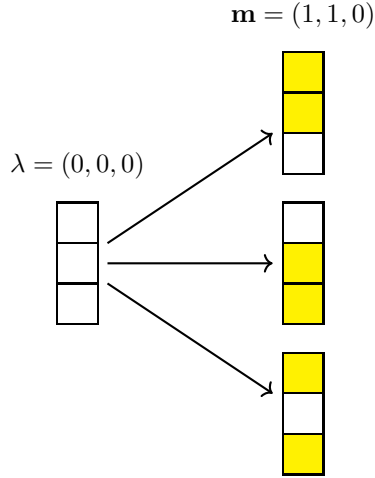}
\caption{The different ways of assembling $(1,1,0)$ by adding a box to $(0,0,0)$. The young diagrams in the lower two rows require reordering. Gray boxes are entries of $\lambda$, yellow boxes are added, and white boxes indicate zeros in the partitions (which we allow here, see Section \ref{sec:degenerations}).}
\label{fig:counting2}
\end{figure}
We record the resulting multiplicities in Table \ref{tab:deg2ints}.
\begin{table}[H]
\centering
\begin{tabular}{  c c c c }
 & $\phi^{\mb{(1,1,0)}}_{1}$ & $\phi^{\mb{(1,0,0)}}_{2}$ & $\phi^{\mb{(0,0,0)}}_{3}$ \\ 
 \hline 
 Mult. of Comp.  & 1 & 2 & 3
\end{tabular}
\caption{The multiplicities of the components of the degeneration in \ref{fig:deg1comps}.}
\label{tab:deg2ints}
\end{table}
Combining the result of these computations gives the equivalence of classes in \ref{fig:deg2final}. 
\begin{figure}[H]
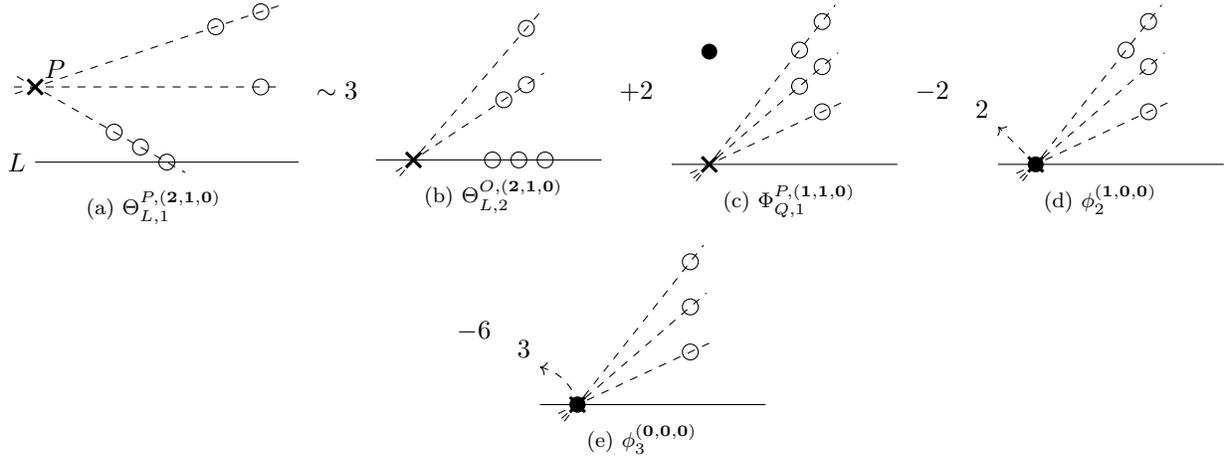

    \centering
    \raisebox{-.4\height}{\subfloat[$\Theta^{P,\mb{(2,1,0)}}_{L,1}$]{\input{figures/pieri/examples/complicated/comp1-1}}}\ $\sim3$~
    \raisebox{-.4\height}{\subfloat[$\Theta^{O,\mb{(2,1,0)}}_{L,2}$]{\input{figures/pieri/examples/complicated/comp2-1}}}\ $+2$~
    \raisebox{-.4\height}{\subfloat[$\Phi^{P,\mb{(1,1,0)}}_{Q,1}$]{\input{figures/pieri/examples/complicated/comp3-0}}}\ $-2$~
    \raisebox{-.4\height}{\subfloat[$\phi^{\mb{(1,0,0)}}_{2}$\label{fig:deg2final2}]{\input{figures/pieri/examples/complicated/comp2-3}}}\ $-6$~
    \raisebox{-.4\height}{\subfloat[$\phi^{\mb{(0,0,0)}}_{3}$]{\input{figures/pieri/examples/complicated/comp3-3}}}
    
    \caption{The equivalence of classes roughly halfway through determining the class in the MS basis of the component in Figure \ref{fig:intcomps1}. }
    \label{fig:deg2final}
\end{figure}
The point here is that we have replaced a class in the sum with a class in the MS basis at the expense of adding classes comprised of schemes which contain nonreduced subschemes of higher length at the point $P$. Unfortunately, we must also accept the possibility of negative signs, and as a result our process is not positive. Since this length is bounded, the process eventually terminates. We continue the process with the locus whose subscheme at $P$ has length two (Figure \ref{fig:deg2final2}).

We start with a locus which is identical except the nonreduced subscheme of length two is supported a fixed point $Q$ distinct from $P$. We degenerate the point $Q$ onto $P$ and obtain two components of the special fiber of the degeneration. The first is the locus of schemes containing a nonreduced subscheme of length two at $P$ and a subscheme of length two collinear with $P$, and the second is the locus of schemes containing a nonreduced subscheme of length three at $P$. See Figure \ref{fig:deg2acomps}. We determine the multiplicities as before to get one and three for the first and second components, respectively. Combining this with the equality in Figure \ref{fig:deg2final}, we arrive at the equality in Figure \ref{fig:deg3final}. Notice that the contribution of the class whose schemes contain subschemes of length three from this last degeneration cancel with those we already have.
\begin{figure}[H]
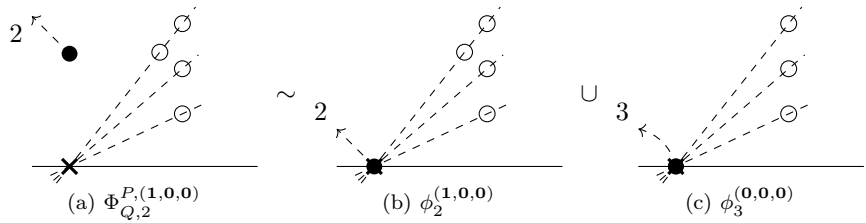

    \centering

    \raisebox{-.4\height}{\subfloat[$\Phi^{P,\mb{(1,0,0)}}_{Q,2}$]{\input{figures/pieri/examples/complicated/comp4-3}}}\ $\sim$
    \raisebox{-.4\height}{\subfloat[$\phi^{\mb{(1,0,0)}}_{2}$]{\input{figures/pieri/examples/complicated/comp2-3}}}\ $\cup$
    \raisebox{-.4\height}{\subfloat[$\phi^{\mb{(0,0,0)}}_{3}$]{\input{figures/pieri/examples/complicated/comp3-3}}}
    
    \caption{The loci of the second degeneration.}
    \label{fig:deg2acomps}
\end{figure}
\begin{figure}[H]
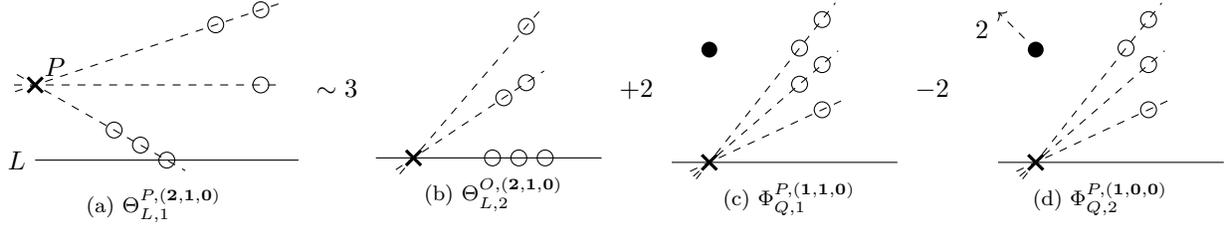

    \centering
    
    \raisebox{-.4\height}{\subfloat[$\Theta^{P,\mb{(2,1,0)}}_{L,1}$]{\input{figures/pieri/examples/complicated/comp1-1}}}\ $\sim3$~
    \raisebox{-.4\height}{\subfloat[$\Theta^{O,\mb{(2,1,0)}}_{L,2}$]{\input{figures/pieri/examples/complicated/comp2-1}}}\ $+2$~
    \raisebox{-.4\height}{\subfloat[$\Phi^{P,\mb{(1,1,0)}}_{Q,1}$]{\input{figures/pieri/examples/complicated/comp3-0}}}\ $-2$~
    \raisebox{-.4\height}{\subfloat[$\Phi^{P,\mb{(1,0,0)}}_{Q,2}$]{\input{figures/pieri/examples/complicated/comp4-3}}}

    \caption{The equivalence in Figure \ref{fig:deg2final} after substituting in the results of the second degeneration. }
    \label{fig:deg3final}
\end{figure}

At this point, we rely on the degeneration of Mallavibarrena and Sols to break up classes of nonreduced subschemes supported at a fixed point into classes whose general member is comprised of distinct points contained in lines. Figure \ref{fig:compfinal} shows the final result of our computation after applying this degeneration.
\begin{figure}[H]
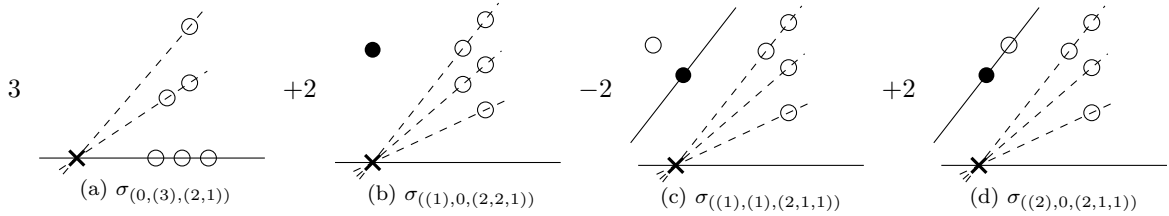

    \centering

    $3$~
    \raisebox{-.4\height}{\subfloat[$\sigma_{(0,(3),(2,1))}$]{\input{figures/pieri/examples/complicated/comp2-1}}}\ $+2$~
    \raisebox{-.4\height}{\subfloat[$\sigma_{((1),0,(2,2,1))}$]{\input{figures/pieri/examples/complicated/comp3-0}}}\ $-2$~
    \raisebox{-.4\height}{\subfloat[$\sigma_{((1),(1),(2,1,1))}$]{\input{figures/pieri/examples/complicated/comp5-1}}}\ $+2$~
    \raisebox{-.4\height}{\subfloat[$\sigma_{((2),0,(2,1,1))}$]{\input{figures/pieri/examples/complicated/comp5-2}}}

    \caption{The results of our hard work: the class of the component in Figure \ref{fig:intcomps1} in the MS basis.}
    \label{fig:compfinal}
\end{figure}

The computation of the class of the second component (Figure \ref{fig:intcomps2}) of the intersection is completely analogous. The result of that computation is contained in Figure \ref{fig:secondcomp}. 

\begin{figure}[H]
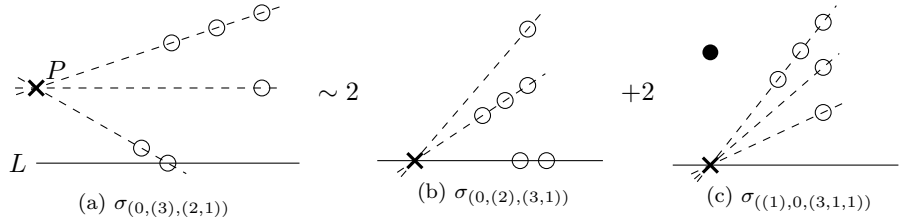

    \centering

    \raisebox{-.4\height}{\subfloat[$\sigma_{(0,(3),(2,1))}$]{\input{figures/pieri/examples/complicated/comp1-2}}}\ $\sim2$~
    \raisebox{-.4\height}{\subfloat[$\sigma_{(0,(2),(3,1))}$]{\input{figures/pieri/examples/complicated/comp6-0}}}\ $+2$~
    \raisebox{-.4\height}{\subfloat[$\sigma_{((1),0,(3,1,1))}$]{\input{figures/pieri/examples/complicated/comp6-1}}}

    \caption{The class of the component in Figure \ref{fig:intcomps2} in the MS basis.}
    \label{fig:secondcomp}
\end{figure}

\section{Degenerations} \label{sec:degenerations}

We need notation for a few loci in the Hilbert scheme before we can describe the degenerations. For that purpose, fix a partition $\mb{m}$ of $N$ of length $r$. For the purpose of this section, we allow partitions to have finitely many zeros and consider them when counting the length. Fix $i$ such that $1 \leq i \leq r$ and such that $\mb{m}_i > \mb{m}_{i+1}$. Let $P$ and $Q$ be distinct points in $\mathbb{P}^2$, let $L$ be a line in $\mathbb{P}^2$, and let $q$ be a positive integer. 

Let $\Theta^{P,\mb{m}}_{L,i}$ be the closure of the locus in $\mathbb{P}^{2[N + r]}$ of schemes $Z$ such that:
\begin{itemize}
\item for each $1 \leq j \leq r$, $Z$ contains a subscheme $Z_j$ of length $\mb{m}_j+1$ collinear with $P$ and spanning a line $L_j$ such that the support of $Z_j$ does not meet $L_k$ for $k \neq j$, and
\item the support of $Z_i$ meets $L$.
\end{itemize}
See Figure \ref{fig:HMLexample} for an example of this locus when $\mb{m} = (3,2,1)$ and $i = 2$. 
\begin{figure}[H]
    \centering
    \input{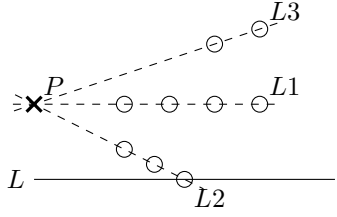}
    \caption{A schematic picture for the scheme $\Theta^{P,(3,2,1)}_{L,2}$.}
    \label{fig:HMLexample}
\end{figure}

Notice that the general scheme in $\Theta_{L,i}^{P,\mb{m}}$ is supported at $N + r$ distinct points. There are $\mb{m}_j + 2$ degrees of freedom for each collinear subscheme except when $j = i$. In this case there are $\mb{m}_i + 1$ degrees of freedom, so  the dimension of $\Theta_{L,i}^{P,\mb{m}}$ is 
$$\mb{m}_i + 1 + \sum_{j \neq i} (\mb{m}_j + 2) = 2r + N - 1.$$

\begin{figure}[H]
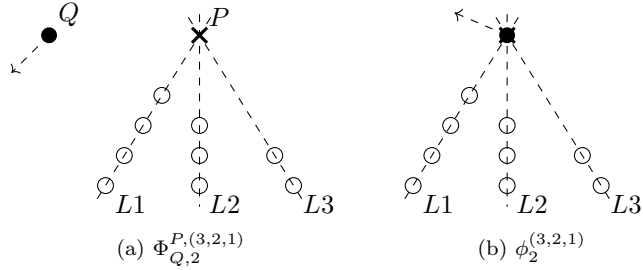

    \centering
    
    \subfloat[$\Phi^{P,(3,2,1)}_{Q,2}$\label{fig:HMQexample1}]{\input{figures/pieri/degenerations/HMQexample}}\qquad
    \subfloat[$\phi^{(3,2,1)}_{2}$\label{fig:HMQexample2}]{\input{figures/pieri/degenerations/HMQexample2}}
    
    \caption{Pictures for the scheme $\Phi^{P,(3,2,1)}_{Q,2}$ and $\phi^{(3,2,1)}_{2}$.}
    \label{fig:HMQexample}
\end{figure}
Let $\Phi_{Q,q}^{P,\mb{m}}$ be the closure of the locus in $\mathbb{P}^{2[N+r+q]}$ of schemes $Z$ such that:
\begin{itemize}
    \item $Z$ contains a nonreduced subscheme $Z_Q$ of length $q$ at $Q$;
    \item for each $1 \leq j \leq r$, $Z$ contains a subscheme $Z_j$ of length $\mb{m}_j + 1$ collinear with $P$ and spanning a line $L_j$ such that the support of any $Z_j$ does not meet $L_j$ for $k \neq j$;
    \item and $Q$ does not lie on $L_j$ for any $j$. 
\end{itemize}
Finally, let $\phi_q^\mb{m}$ be the closure of the locus in $\mathbb{P}^{2[N+q+r]}$ of schemes $Z$ such that:
\begin{itemize}
    \item $Z$ contains a nonreduced subscheme $Z_O$ of length $q$ at the origin;
    \item and for each $1 \leq j \leq r$, $Z$ contains a subscheme $Z_j$ of length $\mb{m}_j + 2$ collinear with the origin and spanning a line $L_j$ such that the support of any $Z_j$ does not meet $L_k$ away from the origin for $k \neq j$.
\end{itemize}
Note that $\phi_q^\mb{m}$ differs from $\Phi_{Q,q}^{P,\mb{m}}$ since $P$ and $Q$ are distinct. See Figure \ref{fig:HMQexample} for an example of each locus when $\mb{m} = (3,2,1)$ and $q = 2$. In either case, there are $q-1$ degrees of freedom defining the length $q$ subscheme at $Q$ and $\mb{m}_i + 2$ degrees of freedom for each collinear subscheme, so  the dimension of $\Phi_{Q,q}^{P,\mb{m}}$ is 
$$(q-1) + \sum_{i=1}^r (\mb{m}_i + 2) = q + 2r + N - 1.$$
Notice that the general scheme in either locus is supported at $N + r + 1$ distinct points. It is important to observe the following fact. 
\begin{lem}
The loci $\phi_q^\mb{m}$ and $\Phi_{Q,q}^{P,\mb{m}}$ are irreducible. 
\end{lem}
\begin{proof}
The loci $\phi_q^\mb{m}$ and $\Phi_{Q,q}^{P,\mb{m}}$ are defined as closures of loci which are isomorphic to the product of a Brian\c{c}on scheme and open sets defining elements of the MS basis, so they are irreducible of the given dimensions.  
\end{proof}

Let $\Sigma_j$ be the set of partitions obtained from some partition $\mb{m}$ by subtracting one $j$ times, at most once from each entry. Expressed as Young diagrams, the partitions obtained from $\mb{m} = (2,1,1)$ are shown in Figure \ref{fig:youngexample}.
\begin{figure}[H]
    \centering
    \input{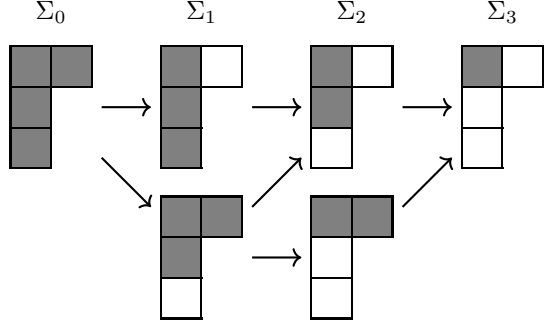}
    \caption{The sets $\Sigma_j$ for $\mb{m} = (2,1,1)$. Gray boxes are present, while white boxes are those which have been removed. }
    \label{fig:youngexample}
\end{figure}
Recall that for the purpose of this section our partitions may have zeros, so that, for instance, the unique partition of $\Sigma_3$ obtained from the partition $\mb{m} = (2,1,1)$ is $(1,0,0)$ and we consider it to be of length three.

Additionally, let $\Sigma_j^i$ be the set of partitions obtained from some partition $\mb{m}$ by subtracting one $j$ times, at most once from each entry, and always beginning with $\mb{m}_i$. This causes no issues because our choice of $i$ was such that $\mb{m}_i > \mb{m}_{i+1}$. For instance, if $\mb{m} = (2,1,1)$ as above, then $\Sigma^2_j = \Sigma^3_j = \Sigma_j$ for $j = 0$ and $j > 1$, but $\Sigma^2_1 = \Sigma^3_1$ is a singleton consisting of only the lower partition shown under $\Sigma_1$ in Figure \ref{fig:youngexample}.

\subsection{The First Degeneration}

Let $O$ be the origin, let $P_t$ be the point $(0,t)$ in the plane with $P = P_1 = (0,1)$, and let $L$ be the line $\set{y = 0}$. Let $F$ be the family in $\mathbb{P}^{2[N+r]} \times \mathbb{C}^*$ such that the fiber $F_t$ over $t \in \mathbb{C}^*$ is $\Theta^{P_t,\mb{m}}_{L,i}$. The closure $\hat{F}$ of $F$ in $\mathbb{P}^{2[N+r]} \times \mathbb{C}$ is a flat family with special fiber $\hat{F}_0$. 

\begin{prop} \label{prop:firstdeg}
The support of the special fiber $\hat{F}_0$ is contained in the union
$$\Theta^{O,\mb{m}}_{L,i} \cup \bigcup_{j = 1}^r \bigcup_{\lambda \in \Sigma^i_j} \phi^\lambda_{j}.$$
\end{prop}
\begin{cor}\label{cor:equiv1}
There is an equivalence of cycles 
$$\left[\Theta^{P,\mb{m}}_{L,i}\right] \sim a \left[\Theta^{O,\mb{m}}_{L,i}\right] + \sum_{j=1}^r \sum_{\lambda \in \Sigma^i_j} c_j^\lambda \left[\phi^\lambda_{j}\right]$$
for nonnegative integers $a$ and $c_j^\lambda$. 
\end{cor}
Notice that every scheme in the component $\Theta^{O,\mb{m}}_{L,i}$ must have a subscheme of length $\mb{m}_i$ contained in $L$.
  
For an example, when the general fiber of the family is $\Theta^{P,(2,1,1)}_{L,2}$, Proposition \ref{prop:firstdeg} says that the special fiber is supported on the irreducible components given by the $\Sigma^2_j$. See Figure \ref{fig:componentexample2} for pictures, but note that, for example, the locus $\phi^{(0,0,0)}_{4}$ does not appear.
\begin{figure}[H]
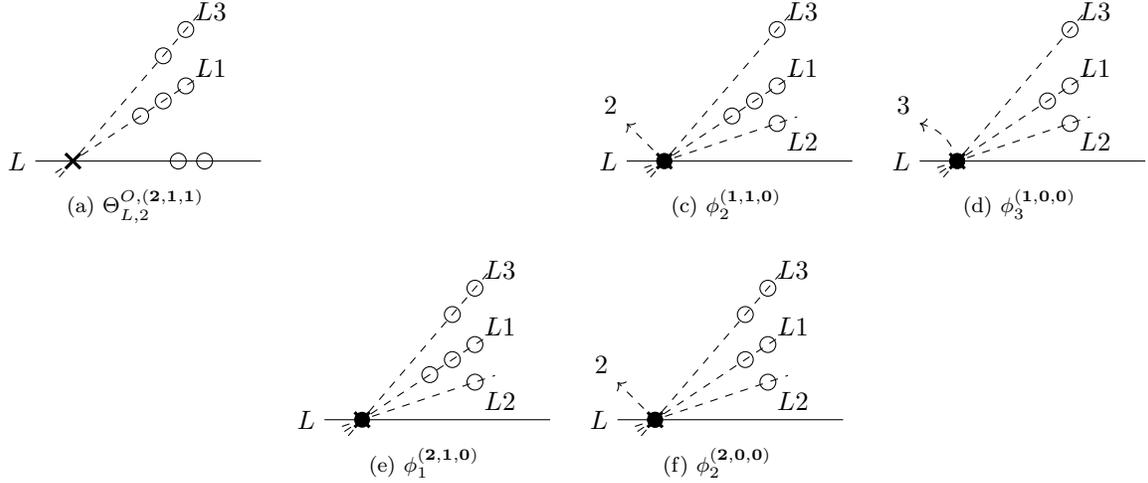

    \centering
    
    \subfloat[$\Theta^{O,\mb{(2,1,1)}}_{L,2}$]{\input{figures/pieri/componentex-deg1/compex1}}\ \,
    \subfloat{\hspace{.24\textwidth}}
    \subfloat[$\phi^{\mb{(1,1,0)}}_{2}$]{\input{figures/pieri/componentex-deg1/compex3}}\ \,
    \subfloat[$\phi^{\mb{(1,0,0)}}_{3}$]{\input{figures/pieri/componentex-deg1/compex4}}\\

    \subfloat[$\phi^{\mb{(2,1,0)}}_{1}$]{\input{figures/pieri/componentex-deg1/compex5}}\ \,
    \subfloat[$\phi^{\mb{(2,0,0)}}_{2}$]{\input{figures/pieri/componentex-deg1/compex6}}\ \,

    \caption{The irreducible components of the support of the special fiber.}
    \label{fig:componentexample2}
\end{figure}

We will use the following lemma in the proof of Proposition \ref{prop:firstdeg}.
\begin{lem}\label{lem:multone}
A general nonreduced subscheme $Z$ of length $k > 0$ supported at the origin has length one along any general line $L$ through the origin.
\end{lem}
\begin{proof}
The general subscheme $Z$ of length $k$ supported at the origin is contained in a smooth curve $\zeta$ of degree $k-1$ \cite{B}. The length $\ell(Z \cap L)$ of the scheme along the line $L$ is the intersection multiplicity $I_O(L,\zeta)$ at the origin. Since the line $L$ is general, this is readily seen to be one.
\end{proof}
\begin{proof}[Proof of Proposition \ref{prop:firstdeg}]
Let $Z_0$ be a general point of the special fiber $\hat{F}_0$ and let $\gamma$ be a curve in $F$ specializing to $Z_0$ such that the general point $\gamma_t$ is contained in $F_t$. The general $\gamma_t$ contains subschemes $\gamma_{t,j}$ of length $\mb{m}_j + 1$ for $1 \leq j \leq r$ such that $\gamma_{t,j}$ spans a line $L_{t,j}$ through $P_t$. Furthermore, the support of the subscheme $\gamma_{t,j}$ meets $L$. Each family of lines $\set{L_{t,j}}$ specializes to a line $L_{0,j}$ through the origin $O$, and the support of $Z_0$ is contained in the union of these lines with the lengths $\ell(Z_0 \cap L_{0,j})$ of $Z_0$ along each line $L_j$ at least $\mb{m}_j + 1$. It follows that if the support of $Z_0$ does not contain the origin, then $Z_0 \in \Theta^{O,\mb{m}}_{L,i}$. Otherwise, $Z_0$ contains a subscheme of length $l > 0$ supported at the origin and subschemes $Z_j$ of length $\mb{m}_j - l_j + 1$ disjoint from and collinear with the origin such that $\sum l_j = l$. 

With that in mind, let $C$ be a component of $\hat{F}_0$ described by a partition $\mb{m}^\prime = (\mb{m}_1 - l_1, \mb{m}_2 - l_2, \ldots, \mb{m}_r - l_r)$ for nonnegative integers $l_i \leq \mb{m}_i$ such that $C$ is contained in the locus $\phi_l^{\mb{m}^\prime}$. It follows that 
$$\dim C \leq l-1 + \sum_{j=1}^r(\mb{m}^\prime_j + 2) = l-1 + \sum_{j=1}^r(\mb{m}_j - l_j + 2) = N + 2r -1$$
with equality holding if and only if $C = \phi_l^{\mb{m}^\prime}$. The general scheme $F_t$ in the family has dimension $N + 2r -1$, so $C$ must too, and therefore $C = \phi_l^{\mb{m}^\prime}$.  

Now, the general scheme $Z \in \phi_l^{\mb{m}^\prime}$ consists of subschemes $Z_j$ of $\mb{m}_j^\prime$ distinct points spanning a line $L_j$ through the origin and a general length $l$ subscheme $Z_O$ at the origin. By Lemma \ref{lem:multone}, the length of $Z_O$ along any of the lines $L_j$ is one, so that
$$\mb{m}_j + 1 \leq \ell(Z\cap L_j) = (\mb{m}_j^\prime + 1) + \ell(Z_O \cap L_j) = (\mb{m}_j - l_j + 1) + 1,$$
and it follows that $l_j \leq 1$. 

Additionally, since $Z$ is the limit of some curve in $F$ whose general member is contained in $F_t$, the subscheme $Z_i \subseteq Z$ is the limit of subschemes $Z_{t,i}$ all meeting $L$ and therefore must meet $L$ also. But then the support of $Z_i$ must meet the origin for general $Z$.
\end{proof}

\subsection{The Second Degeneration}

Now let $q$ be a positive integer, let $Q_t$ be the point $(0,t)$ in the plane, let $Q = Q_1 = (0,1)$ and let $O$ be the origin. Consider the family $F$ in $\mathbb{P}^{2[N+q+r]} \times \mathbb{C}^*$ such that the fiber $F_t$ over $t \in \mathbb{C}^*$ is $\Phi^{O,\mb{m}}_{Q_t,q}$. The closure $\hat{F}$ of $F$ in $\mathbb{P}^{2[N+q+r]} \times \mathbb{C}$ is a flat family with special fiber $\hat{F}_0$. 

\begin{prop} \label{prop:maindeg}
The support of the special fiber $\hat{F}_0$ is contained in the union
$$\bigcup_{j = 0}^r \bigcup_{\lambda \in \Sigma_j} \phi^\lambda_{q+j}.$$
\end{prop}
\begin{cor}\label{cor:equiv2}
There is an equivalence of cycles 
$$\left[\Phi^{O,\mb{m}}_{Q,q}\right] \sim \sum_{j=0}^r \sum_{\lambda \in \Sigma_j} c_j^\lambda \left[\phi^\lambda_{q+j}\right]$$
for nonnegative integers $c_j^\lambda$. 
\end{cor}

For example, when the general fiber of the family is $\Phi^{O,(2,1,1)}_{Q,2}$, Proposition \ref{prop:maindeg} says that the special fiber $\hat{F}_0$ is supported on the irreducible components given by the $\Sigma_j$ in Figure \ref{fig:youngexample}. See Figure \ref{fig:componentexample} for the pictures of each locus, but note that, for instance, the locus $\phi^{(0,0,0)}_{6}$ does not appear.

\begin{figure}[H]
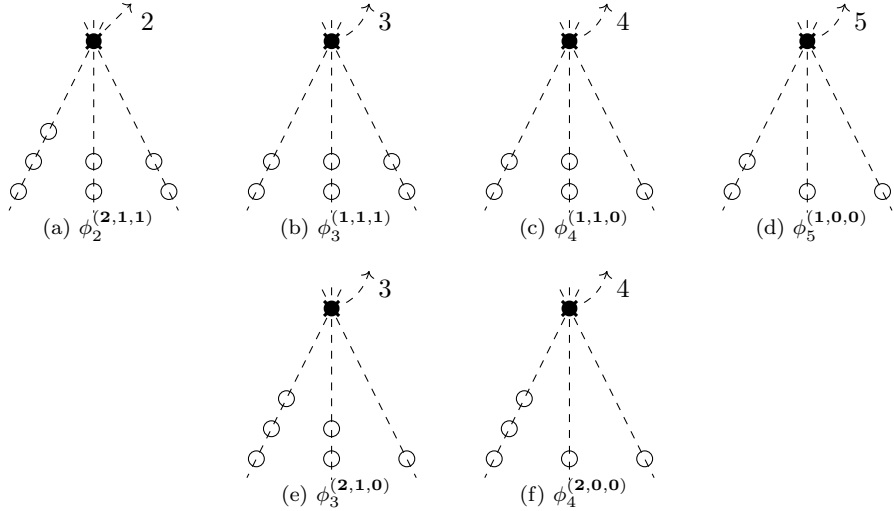

    \centering
    
    \subfloat[$\phi^{\mb{(2,1,1)}}_{2}$]{\input{figures/pieri/componentex-deg2/compex1}}\ \qquad
    \subfloat[$\phi^{\mb{(1,1,1)}}_{3}$]{\input{figures/pieri/componentex-deg2/compex2}}\ \qquad
    \subfloat[$\phi^{\mb{(1,1,0)}}_{4}$]{\input{figures/pieri/componentex-deg2/compex3}}\ \qquad
    \subfloat[$\phi^{\mb{(1,0,0)}}_{5}$]{\input{figures/pieri/componentex-deg2/compex4}}\\
    
    \subfloat[$\phi^{\mb{(2,1,0)}}_{3}$]{\input{figures/pieri/componentex-deg2/compex5}}\ \qquad
    \subfloat[$\phi^{\mb{(2,0,0)}}_{4}$]{\input{figures/pieri/componentex-deg2/compex6}}

    \caption{The irreducible components of the support of the special fiber.}
    \label{fig:componentexample}
\end{figure}

\begin{proof}[Proof of Proposition \ref{prop:maindeg}.]
Let $Z_0$ be a general point of the special fiber $\hat{F}_0$ and let $\gamma$ be a curve in $F$ specializing to $Z_0$ such that the general point $\gamma_t$ is contained in $F_t$. The general $\gamma_t$ contains subschemes $\gamma_{t,i}$ of length $\mb{m}_i + 1$ for $1 \leq i \leq r$ such that $\gamma_{t,i}$ spans a line $L_{t,i}$ through the origin and a subscheme $\gamma_{t,Q_t}$ of length $q$ supported at $Q_t$. The family $L_{t,i}$ for each $i$ specializes to a line $L_{0,i}$ and $Q_t$ specializes to the origin $O$, so that the support of $Z_0$ is contained in $\set{O}\cup \bigcup_{i=1}^r L_{0,i}$. Furthermore, the length of $Z_0$ along any of the $L_{0,i}$ is $\ell(Z_0 \cap L_{0,i}) \geq \mb{m}_i + 1$ and the length of $Z_0$ at $O$ is $\ell(Z_0)_O \geq q$. 

The components of $\hat{F}_0$ therefore consist of loci comprised of schemes containing subschemes whose support is disjoint from the origin, but which are collinear with the origin, of length less than or equal to $\mb{m}_i$ for each $1 \leq i \leq r$ and containing a subscheme of length $q + j$ at the origin for $j \geq 0$. With that in mind, let $C$ be a component of $\hat{F}_0$ described by a partition $\mb{m}^\prime = (\mb{m}_1 - l_1, \mb{m}_2 - l_2, \ldots, \mb{m}_r - l_r)$ for nonnegative integers $l_i$, and let $l$ be the sum of the $l_i$ such that $C$ is the contained in the locus $\phi^{\mb{m}^\prime}_{q+l}$. It follows that 

$$\dim C \leq (q + l - 1) + \sum_{i=1}^r (\mb{m}^\prime_i + 2) = (q + l - 1) + \sum_{i=1}^r (\mb{m}_i - l_i + 2) = q + N + 2r -1.$$

Equality holds if and only if $C = \phi^{\mb{m}^\prime}_{q+l}$. Since the dimension of each fiber $F_t$ has dimension $q + N + 2r - 1$, the dimension of $C$ must be $q + N + 2r - 1$. Hence, the general point $Z_0 \in C$ contains subschemes $Z_{0,i}$ consisting of $\mb{m}_i^\prime + 1$ distinct points collinear with the origin and a general subscheme $Z_{0,O}$ of length $q + l$ supported at the origin.

Assume $Z_0$ is a general point of $C$. As before, let $L_{0,i}$ be the lines through the origin spanned by the subschemes $Z_{0,i}$ and $Z_{0,O}$ be the subscheme of length $q+l$ supported at the origin. By Lemma \ref{lem:multone}, $\ell(Z_{0,O} \cap L_{0,i}) = 1$ for each $i$. On the other hand,
$$\mb{m}_i + 1 \leq \ell(Z_{0} \cap L_{0,i}) = (\mb{m}^\prime_i + 1) + \ell(Z_{0,O} \cap L_{0,i}) = (\mb{m}_i - l_i + 1) + 1,$$
and it follows that $l_i \leq 1$, as desired.
\end{proof}

\subsection{Determining the Multiplicities} \label{sec:coefs}

Fix an integer $j \geq 0$ and let $\lambda$ be a partition of a positive integer $M$ possibly containing zeros as our convention allows. Each component $\phi_{j}^\lambda$ naturally determines a class which we use to determine the multiplicity of $\phi_{j}^\lambda$ in the each degeneration. Let $r$ be the length of $\lambda$, and fix $r$ general points $P_1,\ldots, P_r$ and a general point $Q$. Let $U_j^\lambda$ be the locus of schemes which contain a nonreduced subscheme of length $\lambda_i + 1$ at $P_i$ and which contain a subscheme of length $j$ collinear with $Q$ supported away from the $P_i$'s. See Figure \ref{fig:pairingclass}. We will pair both sides with the class $\tau = [\overline{U_j^\lambda}]$ to determine the multiplicity of $\phi_{j}^\lambda$.

\begin{figure}[h]
    \centering
    \subfloat{\input{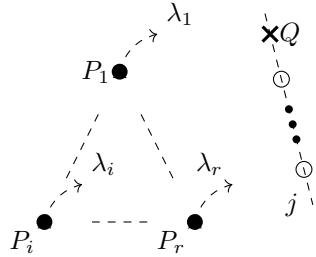}}
    \caption{A picture of the locus $U_j^\lambda$.}
    \label{fig:pairingclass}
\end{figure}

\begin{lem} \label{lem:maincoef}
The intersection $[\phi^\lambda_j] \cdot \tau = 1$, and the intersection $[\phi^{\lambda^\prime}_k] \cdot \tau = 0$ for any other component $\phi^{\lambda^\prime}_k$ of the special fiber of the degenerations in Proposition \ref{prop:firstdeg} and Proposition \ref{prop:maindeg}. 
\end{lem}

In order to prove Lemma \ref{lem:maincoef} we will need the following.

\begin{lem} \label{lem:techlemma}
The intersection in $\mathbb{P}^{2[j]}$ of the class of the locus of schemes of length $j$ supported at a fixed point and the class of the locus of schemes of length $j$ collinear with a fixed point is the class of a single reduced point.
\end{lem}

\begin{figure}[H]
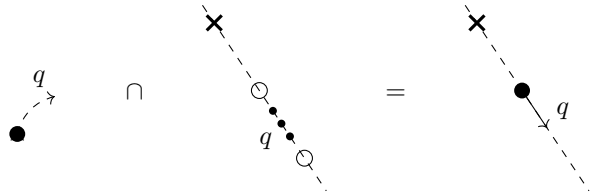

    \centering
    
    \raisebox{-.4\height}{\subfloat{\input{figures/pieri/techlemma/techlemma1}}}\ $\qquad\cap\qquad$
    \raisebox{-.4\height}{\subfloat{\input{figures/pieri/techlemma/techlemma2}}}\ $\qquad=\qquad$
    \raisebox{-.4\height}{\subfloat{\input{figures/pieri/techlemma/techlemma3}}}
    
    \caption{The intersection in Lemma \ref{lem:techlemma}.}
\end{figure}

\begin{proof}
Let the first class be that of locus of subschemes of length $q$ supported at the origin. Let the second class be that of the locus of subschemes collinear with $(-1,0)$. The unique subscheme in the intersection is given by the ideal $(x^j,y)$. There are charts for $\mathbb{P}^{2[j]}$ isomorphic to $\mathbb{A}^{2[j]}$ around this point with coordinates given by coefficients of generators for the ideals
$$(x^j - \gamma_{j-1} x^{j-1} - \cdots - \gamma_1 x - \gamma_0 , y - c_{j-1} x^{j-1} - \cdots - c_1 x - c_0).$$

The first locus is described by the equations $\gamma_i = c_0 = 0$ for all $0 \leq i \leq j-1$. The second locus can be described as those ideals containing the equation of a line through $(-1,0)$, and in particular has equations $c_i = 0$ for $2 \leq i \leq j-1$ and $c_1 = c_0$. It follows that the loci are locally distinct linear spaces of complementary codimension, so they intersect transversally. 
\end{proof}

\begin{proof}[Proof of Lemma \ref{lem:maincoef}.]
To check the first intersection, first set $\phi = \phi_{j}^\lambda$ and $U = U_j^\lambda$ from above. Observe that the points $P_i$ supporting the nonreduced subschemes contained in every scheme in $U$ must each be contained in a single subscheme collinear with the origin, $O$. Similarly, the origin, which supports the length $j$ subscheme common to every scheme in $\phi$, must be contained in the subscheme of $U$ collinear with $Q$. It follows that any scheme $Z \in U \cap \phi$ then contains subschemes of length $\lambda_i + 1$ contained in the lines $\langle O P_i \rangle$ and a subscheme of length $j$ supported at the origin contained in the line $\langle QO \rangle$. There is a unique such scheme, and it suffices to check that the intersection $U \cap \phi$ is transverse at $Z$. 

To that end, $\mathbb{P}^{2[N]}$ has charts centered at $Z$ isomorphic to $\mathbb{A}^{2[j]} \times \mathbb{A}^{2[\lambda_1+1]} \times \cdots \times \mathbb{A}^{2[\lambda_1+1]}$ and it suffices to check transversality locally on each factor. Transversality of the intersection then follows from Lemma \ref{lem:techlemma}.

If $\lambda^\prime \neq \lambda$, then the intersection $\phi^{\lambda^\prime}_k \cap U$, with $U$ as above, is empty since any scheme in the intersection is supported only at the points $P_i$ and the origin. The intersection $\Theta^{O,\mb{m}}_{L,i} \cap U$ is empty since none of the points $P_i$ are contained in $L$. 
\end{proof}

Let $\mb{m}$ and $\lambda$ be partitions such that $\lambda$ is obtained by subtracting one from $j$ entries of $\mb{m}$, as in the statements of Corollaries \ref{cor:equiv1} and \ref{cor:equiv2}. Let $c(\mb{m},\lambda)$ be the number of ways to obtain $\mb{m}$ from $\lambda$ by adding one to $j$ entries of $\lambda$ and then reordering if necessary. 

For example, when $\mb{m} = (2,1,1)$, the $c(\mb{m},\lambda)$ for the possible $\lambda$'s (rf. Figure \ref{fig:youngexample}) are listed in Table \ref{tab:cmlam} below.
\begin{center}
\begin{tabular}{c|c} \label{tab:cmlam}
    $\lambda$ & $c(\mb{m},\lambda)$ \\ \hline
     (1,1,1) & 3\\
     (2,1,0) & 1\\
     (1,1,0) & 2\\
     (2,0,0) & 1\\
     (1,0,0) & 1
\end{tabular}
\end{center}
Notice that, for instance, the first row is three since one can be added to any of the three entries of $\lambda$ and then reordered to obtain $\mb{m}$. Also, as our convention allows, we must consider the partitions $\lambda$ with their entries which are zero.    

\begin{prop} \label{prop:coef}
Let $\mb{m}$ and $\lambda$ be the partitions as in Corollary \ref{cor:equiv1} or Corollary \ref{cor:equiv2}. Let $\ell$ be the number of entries of $\lambda$ equal to $\mb{m}_i - 1$. Let $\hat{\mb{m}}$ be the partition $\mb{m}$ with the entry $\mb{m}_i$ omitted, and let $\hat{\lambda}$ be the partition $\lambda$ with one of its entries equal to $\mb{m}_i - 1$ omitted. The coefficient $c_j^\lambda$ of $\phi^\lambda_j$ in Corollary \ref{cor:equiv1} is
$$c_j^\lambda = \ell \cdot c(\hat{\mb{m}},\hat{\lambda}),$$ and in Corollary \ref{cor:equiv2} is 
$$c_j^\lambda = c(\mb{m},\lambda).$$
\end{prop}

\begin{proof}
By Lemma \ref{lem:maincoef}, the coefficient $c_j^\lambda$ in Corollary \ref{cor:equiv1} is the intersection number $\left[\Theta^{P,\mb{m}}_{L_i}\right] \cdot \tau$ which is $\ell \cdot c(\hat{\mb{m}},\hat{\lambda})$ many schemes. Similarly, the coefficient $c_j^\lambda$ in Corollary \ref{cor:equiv2} is the intersection number $\left[\Phi^{O,\mb{m}}_{Q,q}\right] \cdot \tau$ which is $c(\mb{m},\lambda)$ many schemes. 

It remains to check that these intersections are transverse. This can be done locally since around any point of intersection one has charts $(\mathbb{A}^2)^{[\lambda_1+1]}\times \cdots \times (\mathbb{A}^2)^{[\lambda_r+1]} \times \mathbb{A}^2 \times \cdots \times \mathbb{A}^2$. Lemma \ref{lem:techlemma} shows transversality on the first $r$ factors. In the factors of $\mathbb{A}^2$, one can write charts for both loci and compute the tangent spaces exactly as is done in \cite{ME}. 
\end{proof}

\begin{cor} \label{cor:seconddegcoef}
The coefficient $c^0_\mb{m}$ in Corollary \ref{cor:equiv2} is 1.
\end{cor}
\begin{proof}
This is the case in Proposition \ref{prop:coef} of $\lambda = \mb{m}$, and $c(\mb{m},\mb{m})$ is 1. 
\end{proof}

\begin{lem} \label{lem:firstdegcoef}
The coefficient $a$ in Corollary \ref{cor:equiv1} is $\mb{m}_i + 1$. 
\end{lem}
\begin{proof}
We prove this by intersecting both sides of the equivalence with a class $\tau$ which does not meet any of the other classes on the right hand side but the class $\left[\Theta^{\mb{m}}_{L,i}\right]$. 

Let $\tau$ be the class of the locus of schemes which contain nonreduced subschemes of lengths $\mb{m}_j$ each supported at a fixed general point for $j \neq i$ and which are incident to $\mb{m}_i + 1$ general lines. The intersection of locus with $\Theta^{\mb{m}}_{L,i}$ is one. The intersection of this locus with $\Theta^{P,\mb{m}}_{L,i}$ is $\mb{m}_i + 1$ according to the choice of a point of intersection of $L$ with one of the $\mb{m}_i+1$ fixed general lines.

This determines the coefficient $a$ to be $\mb{m}_i + 1$.
\end{proof}

\subsection{The MS Degeneration}

We finish this section with a final degeneration which is due to Mallavibarrena and Sols \cite{MS}.

Let $L$ be a line in $\mathbb{P}^2$, $Q \in \mathbb{P}^2$ a point, and let $l$ and $q$ be nonnegative integers. Define the subset $H^{L,l}_{Q,q}$ to be the locus of schemes $Z \in \mathbb{P}^{2[l+q]}$ such that
\begin{itemize}
\item the length $\ell(Z \cap L) \geq l$; and
\item the length $\ell(Z)_Q \geq q$.
\end{itemize}

\begin{figure}[H]
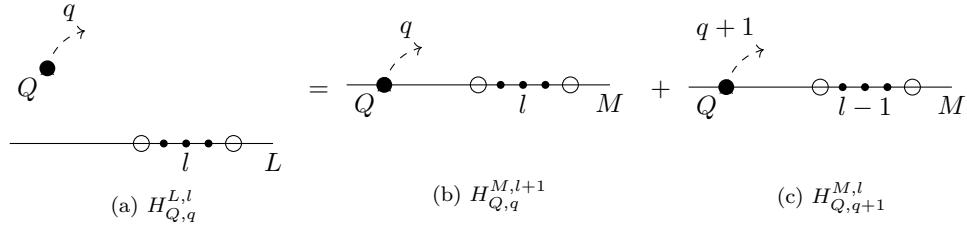

    \centering
    
    \raisebox{-.4\height}{\subfloat[$H^{L,l}_{Q,q}$]{\input{figures/pieri/MSdeg/MSdeg1}}}\ $=$
    \raisebox{-.4\height}{\subfloat[$H^{M,l+1}_{Q,q}$]{\input{figures/pieri/MSdeg/MSdeg2}}}\ $+$
    \raisebox{-.4\height}{\subfloat[$H^{M,l}_{Q,q+1}$]{\input{figures/pieri/MSdeg/MSdeg3}}}

    \caption{The degeneration of Mallavibarrena and Sols.}
    \label{fig:MSdeg}
\end{figure}

\begin{prop}\label{prop:MSdeg}
Let $Q$ be a point of $\mathbb{P}^2$, $L$ a line not containing $Q$, and $M$ a line containing $Q$. Let $l$ and $q$ be nonnegative integers. There is an equivalence of classes on $\mathbb{P}^{2[l+q]}$
$$\left[H^{L,l}_{Q,q}\right] \sim \left[H^{M,l}_{Q,q+1}\right] + \left[H^{M,l+1}_{Q,q}\right].$$
\end{prop}

For a proof, see Mallavibarrena and Sols \cite{MS}. A schematic picture can be found in Figure \ref{fig:MSdeg}. We obtain the obvious corollary.

\begin{cor}\label{cor:MS1}
There is an equivalence of classes
$$\left[H^{M,1}_{Q,q}\right] = \sum_{i=1}^q (-1)^{i+1} \left[H^{L,i}_{Q,q-i}\right] + (-1)^{q-1}\left[H^{M,q}_{Q,1}\right].$$
\end{cor}

A slight modification of the degeneration in Proposition \ref{prop:MSdeg} will allow us to resolve the class of loci appearing naturally in the intersections $H \cdot \sigma$. To that end, let $Q$ be a point of $\mathbb{P}^2$, $L$ be a line containing a fixed point $P$, and $l,q$ be nonnegative integers. Define the locus $G^{P,L,l}_{Q,q}$ to be the schemes $Z \in \mathbb{P}^{2[l + q + 1]}$ such that
\begin{itemize}
\item $P \in Z$;
\item the length $\ell(Z \cap L) \geq l$; and
\item the length $\ell(Z)_Q \geq q$.
\end{itemize}
Notice that $G^{P,L,l}_{Q,q}$ differs from $H^{L,l}_{Q,q}$ since its members contain the additional fixed point $P \in L$. By fixing a point on $L$ in Proposition \ref{prop:MSdeg}, we obtain the following corollary.

\begin{cor}\label{cor:MS2}
Let $P$ and $Q$ be distinct fixed points in $\mathbb{P}^2$, let $L$ be the line they span, and let $R$ be a fixed point not contained in $L$. There is an equivalence of classes
$$\left[G_{Q,1}^{P,L,l+2}\right] = \sum_{i=0}^{l} (-1)^{i} \left[G_{R,i}^{P,L,l+1-i}\right].$$
\end{cor}

\section{The Algorithm}

In this section, we describe how to compute the class of the intersection of the divisor $\sigma_{(0,1,(N-1))}$ with any other class $\sigma_{\alpha}$ for a triple of partitions $\alpha = (\mb{a},\mb{b},\mb{c})$ in the Hilbert scheme $\mathbb{P}^{2[N]}$. We begin with some immediate observations and reductions.

First, we fix a general representative $U_\alpha$ for $\sigma_{\alpha}$ and $H$ for $\sigma_{(0,1,(N-1))}$. The intersection of these two representatives is generically transverse as long as the lines and points defining $U_\alpha$ and the line defining $H$ are chosen generally with respect to each other, as is the case. The irreducible components of the intersection are thus easy to identify: the point moving on the line defining $H$ must satisfy one of the three different types of conditions defining $U_\alpha$ which we call type A, B, and C, respectively. Type A intersections contain as a fixed point the intersection point of the line defining $H$ and one of the lines through a fixed point defining $U_\alpha$. Type B intersections contain as a fixed point the intersection point of the line defining $H$ and one of the fixed lines defining $U_\alpha$ which does not pass through a fixed point defining $U_\alpha$. Type C intersections occur when one of the points on a moving line through $Q$ defining $U_\alpha$ resides on the line defining $H$. See Figure \ref{fig:threetypes} for schematic diagrams of these intersections.

\begin{figure}[H]
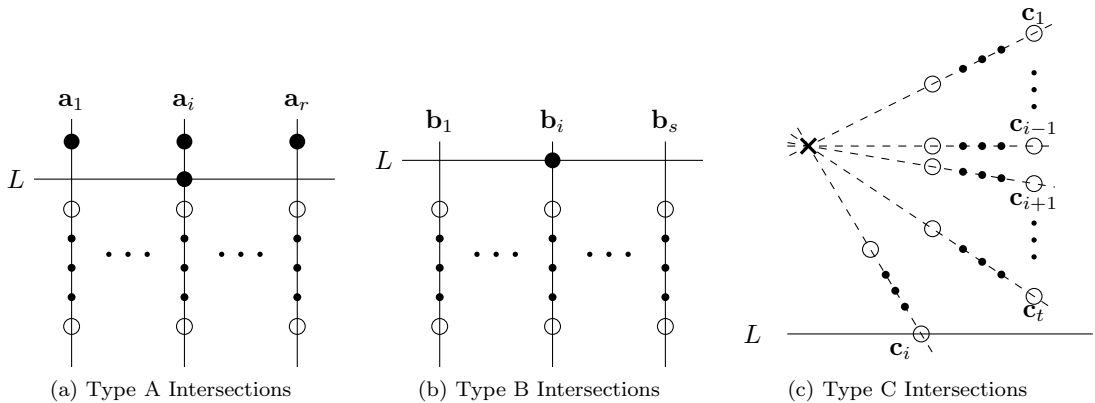

    \centering
    
    \subfloat[Type A Intersections\label{fig:typeA}]{\input{figures/pieri/types/type1}}\ \,
    \subfloat[Type B Intersections\label{fig:typeB}]{\input{figures/pieri/types/type2}}\ \,
    \subfloat[Type C Intersections\label{fig:typeC}]{\input{figures/pieri/types/type3}}\ \,
    
    \caption{The three different types of intersections.}
    \label{fig:threetypes}
\end{figure}

The class $\sigma_{(\mb{a},\mb{b},\mb{c})} \cdot \sigma_{(0,1,(N-1))}$ is the sum of the classes of each component of the intersection of $H$ with $U_\alpha$. It remains to compute these classes. We will dispatch the easiest case first. 

\subsection{Type B}

The simplest case is when the point moving on the line defining $H$ meets any of the fixed lines defining $U_\alpha$ coming from the partition $\mb{b}$ as shown in Figure \ref{fig:typeB}. The classes of these components are already elements of the MS basis, and there is very little work to be done. 

If $\mb{b} = (\mb{b}_1,\ldots,\mb{b}_s)$, then the resulting class describing the intersection along these types is given by the sum
$$\sum_{i = 1}^s \sigma_{(\mb{a}^{i},\widehat{\mb{b}^{i}},\mb{c})}$$
where $\widehat{\mb{b}^{i}}$ is $\mb{b}$ with $\mb{b_i}$ omitted and $\mb{a}^{i}$ is $\mb{a}$ with $\mb{b}_i$ inserted. Notice that if an entry of $\mb{b}$ is repeated then there are repeated summands.

\subsection{Type A}

This occurs when the point moving on the line defining $H$ meets any of the lines in $U_\alpha$ through a fixed point of $U_\alpha$ as defined by $\mb{a}$ as shown in Figure \ref{fig:typeA}. Because the resulting locus contains schemes with two fixed points on the same line, possibly with or without remaining moving points on that line, we must resolve the class of this locus into elements of the MS basis. 

\begin{figure}[H]
    \centering
    \input{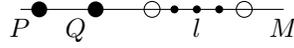}
    \caption{A locus of type A, or $G^{P,L,l+2}_{Q,1}$.}
    \label{fig:typeA2}
\end{figure}

For this we turn to the degeneration of Mallavibarrena and Sols given in Proposition \ref{prop:MSdeg}. Specifically, Corollary \ref{cor:MS2} allows us to write the class of the locus $G^{P,L,l+2}_{Q,1}$ (Figure \ref{fig:typeA2}) of the intersection as an alternating sum of the classes of loci consisting of a fixed point $P$ with some number of points on a line through $P$ and a nonreduced subscheme supported at a point away the line. The class of these loci can then be written in the MS basis by repeatedly applying Corollary \ref{cor:MS1}.

\subsection{Type C}

Let $\mb{c}^\prime$ be the partition given by $\mb{c} - 1 := (\mb{c}_1 - 1, \mb{c}_2 -1, \ldots, \mb{c}_t -1)$. This shift is necessary to align with our notation from Section \ref{sec:degenerations}. Any locus of Type C as shown in Figure \ref{fig:typeC} is consequently the locus $\Theta^{Q,\mb{c}^\prime}_{L,i}$ for general line $L$ defining the locus $H$ and general point $Q$ defining $U_\alpha$.

The general steps are as follows.
\begin{enumerate}
\item We apply the degeneration from Proposition \ref{prop:firstdeg} and Corollary \ref{cor:equiv1} to write the class $[\Theta^{Q,\mb{c}^\prime}_{L,i}]$ as a sum of classes of loci of type $\phi^\lambda_j$. We determine the coefficients $c^\lambda_j$ in Corollary \ref{cor:equiv1} by Proposition \ref{prop:coef}, see Section \ref{sec:coefs}.
\item We then resolve the classes $[\phi^\lambda_j]$ in order of increasing $j$. To each locus we apply the degeneration in Proposition \ref{prop:maindeg} and Corollary \ref{cor:equiv2}. In doing so, we replace $[\phi^\lambda_j]$ with a class of $[\Phi^\lambda_{Q,j}]$ at the cost of adding additional classes of loci $\phi^{\lambda^\prime}_k$ for $k > j$. We again determine the coefficients $c_k^{\lambda^\prime}$ in Corollary \ref{cor:equiv2} by Proposition \ref{prop:coef}.
\item Repeat step 2 increasing the length of the subscheme common to all schemes in the locus at each step. Since this is bounded above by the total number of points $C$, eventually this terminates in the class $[\phi^{(0,0,\ldots,0)}_{C}] = [\Phi^{(0,0,\ldots,0)}_{Q,C}]$.
\item We replace each class $[\Phi^\lambda_{Q,j}]$ with a sum of classes in the MS basis using the degeneration of Mallavibarrena and Sols in Proposition \ref{prop:MSdeg} and Corollary \ref{cor:MS1}.
\end{enumerate} 

\subsection{Ending Remarks} \label{sec:remarks}

Finally, we have a conjecture for a closed form for the intersection of $H$ with an MS basis class consisting of only moving lines. 
\begin{conj}\label{conj:pieri}
The intersection of $H$ with a general MS locus $(0,0,\mb{m})$ along a moving line with $\mb{m}_i$ many points on it is a weighted sum over classes of the loci $\Theta^\mb{m}_{L,i}$ and $\phi^\lambda_j$ indexed by only $\lambda$'s obtained by subtracting an integer $0 \leq x \leq \mb{m}_i$ from the entry $\mb{m}_i$ in $\mb{m}$.
\end{conj}

As a remark, the conjecture can be phrased completely combinatorially. Beginning with $\mb{m}$, form a directed acyclic graph (DAG) as follows. The root node is labeled $\mb{m}$ and has an edge to a node labeled with any partition $\lambda$ in $\Sigma_j^i$ (see the beginning of Section \ref{sec:degenerations}). The weight $w(e)$ of an edge $e$ is the multiplicity of the corresponding component as given in Proposition \ref{prop:coef}. For each of the new nodes, add an edge to a node labeled with any partition obtained from it by subtracting one from any number of its entries. The weight of one of these edges is $c(\lambda,\lambda^\prime)$ where the edge goes from $\lambda$ to $\lambda^\prime$. Repeat this process for every new node. The weight $w(p)$ of any path $p = \set{e_1,\ldots,e_n}$ from the root node $\mb{m}$ to any node $\lambda$ in the DAG is
$$w(p) = (-1)^{n+1} w(e_1) \cdots w(e_n).$$
\begin{conj}
Let $\lambda$ be any node which is not obtained from $\mb{m}$ by subtracting an integer $0 \leq x \leq \mb{m}_i$ from $\mb{m}_i$. If $S$ is the set of all paths from $\mb{m}$ to $\lambda$, then
$$\sum_{p \in S} w(p) = 0.$$
\end{conj}
We end with the example in Figure \ref{fig:DAG} for $\mb{m} = (2,1,1)$ and $i=1$. In particular, note that the sum of the weights of all paths to $(0,0,0)$ and $(1,0,0)$ are zero. 

\begin{figure}[H]
\centering
\input{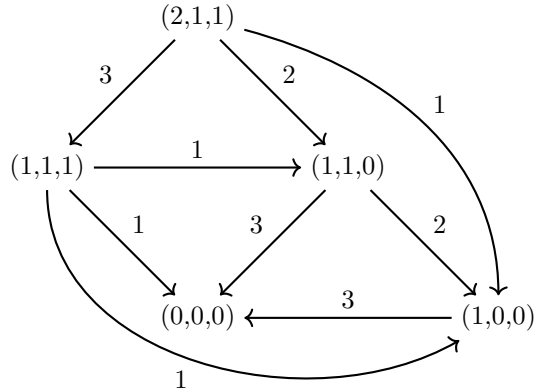}
\caption{The DAG for $\mb{m} = (2,1,1)$ with $i = 1$.}
\label{fig:DAG}
\end{figure}

\bibliographystyle{plain}
\bibliography{cites}
\end{document}